\documentclass[11pt]{amsart}

\usepackage{geometry}
\usepackage{graphicx}
\usepackage{amsmath}
\usepackage{amsthm}
\usepackage{amsfonts}
\usepackage{amssymb, amscd}

\usepackage[latin1]{inputenc}

\usepackage {latexsym,times}

\usepackage{epsfig,color}

\usepackage[all]{xy}

\newcommand{\K}{\mathbb{K}}

\newtheorem{theorem}{Theorem}[section]

\newtheorem{lemma}[theorem]{Lemma}
\newtheorem{proposition}[theorem]{Proposition}

\newtheorem{definition}[theorem]{Definition}

\newtheorem{example}[theorem]{Example}

\newtheorem{remark}[theorem]{Remark}

\input{epsf.sty}

\begin{document}

\title[Hom-alternative algebras and Hom-Jordan algebras]
{Hom-alternative algebras and Hom-Jordan algebras}

\author{Abdenacer Makhlouf}
\address{Universit\'e de Haute-Alsace, Laboratoire de
Math\'ematiques, Informatique et Applications, 4 rue des Fr\`{e}res
Lumi\`{e}re, 68093 Mulhouse, France  } \email{Abdenacer.Makhlouf@uha.fr}

\begin{abstract}
The purpose of this paper is to  introduce Hom-alternative algebras
and Hom-Jordan algebras. We discuss some of their properties and provide
construction procedures using ordinary alternative algebras or Jordan algebras.
Also, we show that a polarization of Hom-associative algebra leads to Hom-Jordan algebra.
\end{abstract}

\date{}
\dedicatory{Dedicated to Amine Kaidi on his 60th birthday}

\maketitle

\section*{Introduction}
Hom-algebraic structures are algebras where the identities defining the structure are twisted
by a homomorphism. They have been intensively investigated in the literature recently.
The Hom-Lie algebras  were introduced and discussed   in
\cite{HLS,LS1,LS2,LS3}, motivated by quasi-deformations of Lie algebras of vector fields, in particular $q$-deformations of Witt and Virasoro
algebras. Hom-associative algebras
were
introduced in \cite{MS},  where it is shown that
 the commutator bracket of a Hom-associative algebra
gives rise to a Hom-Lie algebra and where a classification of Hom-Lie admissible
algebras is established.
 Given a Hom-Lie algebra, there is a universal enveloping Hom-associative
 algebra (see \cite{Yau:EnvLieAlg}).
Dualizing  Hom-associative algebras, one can define Hom-coassociative
coalgebras, Hom-bialgebras and Hom-Hopf algebras which were introduced in
( \cite{HomAlgHomCoalg,HomHopf}), see also \cite{Canepl2009,Yau:YangBaxter,
Yau:YangBaxter2,Yau:ClassicYangBaxter,Yau:HomQuantumGrp1,Yau:HomQuantumGrp2}. It is
shown in \cite{Yau:comodule} that the universal enveloping Hom-associative
algebra carries a structure of
Hom-bialgebra. See also \cite{AmmarMakhlouf2009,AMS2009,
FregierGohr1,FregierGohr2,Gohr,HomDeform,Yau:homology}
for other works on twisted algebraic structures.

The purpose of this paper is to  introduce Hom-alternative algebras
and Hom-Jordan algebras which are twisted version of the ordinary alternative
algebras and Jordan algebras. We discuss some of their properties and provide
construction procedures using ordinary alternative algebras or Jordan algebras.
Also, we show that a polarization of Hom-associative algebra leads to Hom-Jordan algebra.

In the first Section of this paper we introduce Hom-alternative algebras and
study their properties. In particular, we define a twisted version of the associator and show that it is an alternating function of its arguments. The second Section is devoted to construction of
Hom-alternative algebras. We show that an ordinary alternative algebra and one
of its algebra endomorphisms lead to a Hom-alternative algebra where the twisting
map is actually the algebra endomorphism. This process was
  introduced in \cite{Yau:homology}
 for Lie and associative algebras and more generally to $G$-associative algebras
 (see \cite{MS} for this class of algebras) and generalized to coalgebras in \cite{HomAlgHomCoalg},
 \cite{HomHopf} and to $n$-ary algebras of Lie and associative types in
 \cite{AMS2009}. We derive examples of Hom-alternative algebras from
  4-dimensional alternative algebras which are
 not associative and from algebra of octonions. The last Section is dedicated
 to Jordan algebras. We introduce a notion of Hom-Jordan algebras and show that
 it fits with the Hom-associative structure, that is a Hom-associative algebra
 leads to Hom-Jordan algebra by polarization. Also, we provide a way to construct a Hom-Jordan
 algebra starting from an ordinary Jordan algebra and an algebra endomorphism.

\section{Definitions and properties}
Throughout this paper $\mathbb{K}$ is a field of characteristic 0
and $V$ be a $\K$-linear space.

First, we recall the notion of Hom-associative algebra introduced in
\cite{MS} and provide an example.
\begin{definition}
A \textbf{\emph{Hom-associative algebra}} over $V$ is a triple $( V, \mu,
\alpha) $ where  $\mu : V\times V \rightarrow V$ is a bilinear map
and $\alpha:V \rightarrow V$ is a linear map, satisfying
\begin{equation}\label{Hom-ass}
\mu(\alpha(x), \mu (y, z))= \mu (\mu (x, y), \alpha (z)).
\end{equation}
\end{definition}
\begin{example}\label{example1ass}
Let $\{e_1,e_2,e_3\}$  be a basis of a $3$-dimensional linear space
$V$ over $\K$. The following multiplication $\mu$ and linear map
$\alpha$ on $V$ define Hom-associative algebras over $\K^3${\rm :}
$$
\begin{array}{ll}
\begin{array}{lll}
 \mu ( e_1,e_1)&=& a\ e_1, \ \\
\mu ( e_1,e_2)&=&\mu ( e_2,e_1)=a\ e_2,\\
\mu ( e_1,e_3)&=&\mu ( e_3,x_1)=b\ e_3,\\
 \end{array}
 & \quad
 \begin{array}{lll}
\mu ( e_2,e_2)&=& a\ e_2, \ \\
\mu ( e_2, e_3)&=& b\ e_3, \ \\
\mu ( e_3,e_2)&=& \mu ( e_3,e_3)=0,
  \end{array}
\end{array}
$$

$$  \alpha (e_1)= a\ e_1, \quad
 \alpha (e_2) =a\ e_2 , \quad
   \alpha (e_3)=b\ e_3,
$$
where $a,b$ are parameters in $\K$. The algebras are not associative
when $a\neq b$ and $b\neq 0$, since
$$\mu (\mu (e_1,e_1),e_3))- \mu ( e_1,\mu
(e_1,e_3))=(a-b)b e_3.$$

\end{example}
Now, we introduce the notion of Hom-alternative algebra.
\begin{definition}
A \textbf{\emph{left Hom-alternative}} algebra (resp. \textbf{right
Hom-alternative algebra}) is a triple $(V,\mu,\alpha )$
consisting of a $\K$-linear space $V$, a linear map $\alpha: V
\rightarrow V$ and a multiplication $\mu: V\otimes V \rightarrow V$
satisfying the left Hom-alternative identity, that is
\begin{equation}\label{HomLeftAlternative}
\mu (\alpha (x),\mu(x, y))=\mu(\mu(x, x) , \alpha (y)),
\end{equation}
respectively, right Hom-alternative identity, that is
\begin{equation}\label{HomRightAlternative}
\mu (\alpha(x),\mu(y, y))=\mu(\mu(x,y) ,\alpha( y)),
\end{equation}
for any $x,y$ in $V$.

A Hom-alternative algebra is one which is both left and right
Hom-alternative algebra.
\end{definition}

\begin{remark}
Any Hom associative algebra is a Hom-alternative algebra.
\end{remark}
\begin{definition}Let $\left( V,\mu ,\alpha \right) $ and $\left( V^{\prime },\mu
^{\prime },\alpha^{\prime }\right) $ be two Hom-alternative
algebras. A linear map $f\ :V\rightarrow V^{\prime }$ is said to be
a
\emph{morphism of Hom-alternative algebras} if%
$$  \mu ^{\prime }\circ (f\otimes f)=f\circ \mu
\text{,} \qquad f\circ \alpha=\alpha^{\prime }\circ f .
$$
\end{definition}

Let $\mathfrak{as_\alpha}$  denotes the
\textbf{\emph{Hom-associator}} associated to a Hom-algebra
$(V,\mu,\alpha)$ where $V$ is the linear space, $\mu$ the
multiplication and $\alpha$ the twisting map. The Hom-associator is
a trilinear map defined for any $x,y,z\in V$ by
\begin{equation}\label{HomAssociator}
\mathfrak{as}_\alpha(x,y,z)=\mu (\alpha(x),\mu (y, z))-\mu(\mu(x,
y), \alpha(z)).
\end{equation}
The condition \ref{HomLeftAlternative} (resp.
\ref{HomRightAlternative}) may be written using Hom-associator
respectively
\begin{eqnarray*}
 \mathfrak{as}_\alpha(x,x,y)=0 ,\quad \mathfrak{as}_\alpha(y,x,x)=0.
 \end{eqnarray*}

By linearization, we have the following equivalent definition of
left and right Hom-alternative algebras.

\begin{proposition}
A triple  $(V,\mu,\alpha)$ is a left Hom-alternative algebra (resp.
right alternative algebra) if and only if the identity
\begin{equation}\label{HomLeftAlternativeLineariz}
\mu (\alpha (x),\mu(y, z))-\mu(\mu(x, y),\alpha( z))+\mu (\alpha
(y), \mu(x, z))-\mu(\mu(y, x), \alpha (z))=0.
\end{equation}
respectively,
\begin{equation}\label{HomRightAlternativeLineariz}
\mu (\alpha(x),\mu (y, z))-\mu (\mu (x, y), \alpha(z))+\mu
(\alpha(x), \mu (z, y))-\mu (\mu (x, z), \alpha(y))=0.
\end{equation}
holds.
\end{proposition}
\begin{proof}
We assume that,  for any $x,y,z\in V$,
$\mathfrak{as}_\alpha(x,x,z)=0 $ (left alternativity), then we
expand $\mathfrak{as}_\alpha(x+y,x+y,z)=0 $.

The proof for right Hom-alternativity is obtained by expanding
$\mathfrak{as}_\alpha(x,y+z,y+z)=0 .$

Conversely, we set $x=y$ in (\ref{HomLeftAlternativeLineariz}), respectively $y=z$ in (\ref{HomRightAlternativeLineariz}).
\end{proof}

  \begin{remark} The multiplication could be considered as a linear map
  $\mu : V \otimes V \rightarrow V$, then the condition
  (\ref{HomLeftAlternativeLineariz}) (resp. (\ref{HomRightAlternativeLineariz})) writes
  \begin{equation}\label{HomLeftAlternativeLineariz2}
\mu \circ (\alpha\otimes \mu-\mu \otimes \alpha)\circ (id^{\otimes
3} +\sigma_{1})=0,
\end{equation}
respectively
\begin{equation}\label{HomRightAlternativeLineariz2}
\mu \circ (\alpha\otimes \mu-\mu \otimes \alpha)\circ (id^{\otimes
3}+\sigma_{2})=0.
\end{equation}
where $id$ stands for the identity map and $\sigma_{1}$ and
$\sigma_{2}$ stands for trilinear maps defined for any $x_1,x_2, x_3\in V$ by
\begin{eqnarray*}
\sigma_{1}(x_1\otimes x_2\otimes x_3)=x_2\otimes x_1\otimes x_3,\\
\sigma_{2}(x_1\otimes x_2\otimes x_3)=x_1\otimes x_3\otimes
x_2.
\end{eqnarray*} 

In terms of associators, the identities
(\ref{HomLeftAlternativeLineariz}) (resp.
(\ref{HomRightAlternativeLineariz})) are equivalent respectively to 
\begin{equation}\label{Alt0}\mathfrak{as}_\alpha+\mathfrak{as}_\alpha\circ\sigma_{1}=0 \quad
\text{ and  }\
\mathfrak{as}_\alpha+\mathfrak{as}_\alpha\circ\sigma_{2}=0.
\end{equation}
  \end{remark}
  Hence, for any $x,y,z\in V$, we have
  \begin{equation}\label{Alt1}
  \mathfrak{as}_\alpha(x,y,z)=-\mathfrak{as}_\alpha(y,x,z)\quad \text{ and
  }\quad \mathfrak{as}_\alpha(x,y,z)=-\mathfrak{as}_\alpha(x,z,y).
  \end{equation}
  We have also the following property.
\begin{lemma}\label{HomAlternating1}
Let $(V,\mu,\alpha)$ be an alternative Hom-algebra. Then
\begin{equation}\label{Alt1}
\mathfrak{as}_\alpha(x,y,z)=-\mathfrak{as}_\alpha(z,y,x).
\end{equation}
\end{lemma}
\begin{proof}
Using (\ref{Alt0}), we have
\begin{eqnarray*}
\mathfrak{as}_\alpha(x,y,z)+\mathfrak{as}_\alpha(z,y,x)&=&
-\mathfrak{as}_\alpha(y,x,z)-\mathfrak{as}_\alpha(y,z,x)\\
\ &=&0.
\end{eqnarray*}
\end{proof}
\begin{remark}
The identities (\ref{Alt0},\ref{Alt1}) lead to the fact that an
algebra is Hom-alternative if and only if the Hom-associator
$\mathfrak{as}_\alpha(x,y,z)$ is an alternating function of its
arguments, that is
$$\mathfrak{as}_\alpha(x,y,z)=-\mathfrak{as}_\alpha(y,x,z)=
-\mathfrak{as}_\alpha(x,z,y)=-\mathfrak{as}_\alpha(z,y,x).$$
\end{remark}
\begin{proposition}
A Hom-alternative algebra is Hom-flexible, that is
$\mathfrak{as}_\alpha(x,y,x)=0$.
\end{proposition}
\begin{proof}
Using  lemma \ref{HomAlternating1}, we have
$\mathfrak{as}_\alpha(x,y,x)=-\mathfrak{as}_\alpha(x,y,x)$.
Therefore, $\mathfrak{as}_\alpha(x,y,x)=0$.
\end{proof}

\begin{proposition}
Let $(V,\mu,\alpha)$ be a Hom-alternative algebra and $x,y,z\in V$.

If $x$ and $y$ anticommute, that is $\mu (x,y)=-\mu (y,x)$, then we
have
\begin{equation}
\mu (\alpha (x),\mu (y,z))=-\mu (\alpha (y),\mu (x,z),
\end{equation}
and
\begin{equation}
\mu (\mu (z,x),\alpha (y))=-\mu (\mu (z,y),\alpha (x)).
\end{equation}
\end{proposition}
\begin{proof}
The left alternativity leads to
\begin{equation}
\mu (\alpha (x),\mu(y, z))-\mu(\mu(x, y),\alpha( z))+\mu (\alpha
(y), \mu(x, z))-\mu(\mu(y, x), \alpha (z))=0.
\end{equation}
Since $\mu(x, y)=-\mu(y, x)$, then the previous identity becomes
\begin{equation}
\mu (\alpha (x),\mu(y, z))+\mu (\alpha (y), \mu(x, z))=0.
\end{equation}
Similarly, using the right alternativity and the assumption of
anticommutativity, we get the second identity.
\end{proof}

  \begin{remark}
  A subalgebra of an alternative algebra $(V,\mu,\alpha)$ is given by a
   subspace
  $W$ of $V$ such that for any $x,y\in W$, we have  $\mu(x,y)\in W$ and
  $\alpha (x)\in W$. The multiplication and the twisting map being the same.
  The notions of ideal, quotient algebra are defined as usual and similarly.

\end{remark}

\section{Construction theorem and Examples}
In this section, we provide a way to construct Hom-alternative
algebras starting from an alternative algebra and an algebra endomorphism. This procedure was applied to associative
algebras, $G$-associative algebras and Lie algebra in
\cite{Yau:homology}. It was  extended to coalgebras in
\cite{HomAlgHomCoalg} and to $n$-ary algebras of Lie type
respectively associative type in \cite{AMS2009}.

\begin{theorem}\label{thmConstrHomAlt}
Let $(V,\mu)$ be a left  alternative algebra (resp. a right
alternative algebra ) and  $\alpha : V\rightarrow V$ be an
 algebra endomorphism. Then $(V,\mu_\alpha,\alpha)$,
where $\mu_\alpha=\alpha\circ\mu$,  is a left  Hom-alternative
algebra (resp. right Hom-alternative algebra).

Moreover, suppose that  $(V',\mu')$ is another left  alternative algebra (resp. a right
alternative algebra )  and $\alpha ' : V'\rightarrow V'$ is an algebra
endomorphism. If $f:V\rightarrow V'$ is an algebras morphism that
satisfies $f\circ\alpha=\alpha'\circ f$ then
$$f:(V,\mu_\alpha,\alpha)\longrightarrow (V',\mu'_{\alpha '},\alpha ')
$$
is a morphism of left  Hom-alternative algebras (resp. right Hom-alternative algebras).
\end{theorem}
\begin{proof}
We show that $(V,\mu_\alpha,\alpha)$ satisfies the left
Hom-alternative identity (\ref{HomLeftAlternative}). Indeed
\begin{align*}
\mu_\alpha(\alpha(x)\otimes \mu_\alpha (x\otimes y))&=
\alpha(\mu(\alpha(x)\otimes \alpha(\mu (  x)\otimes  y)))\\
&=\alpha(\mu(\alpha(x)\otimes \mu ( \alpha (x)\otimes  \alpha (y)))\\
&=\alpha(\mu( \mu (\alpha(x)\otimes \alpha (x))\otimes  \alpha (y)))\\
&=\alpha(\mu( \alpha(\mu (x\otimes x))\otimes  \alpha (y)))\\
&= \mu_\alpha (\mu_\alpha (x\otimes x)\otimes \alpha (y)).
\end{align*}
The proof for right alternativity is obtained similarly.

the second assertion follows from
$$f\circ \mu_\alpha=f\circ \alpha \circ \mu=\alpha'\circ
f \circ \mu =\alpha'\circ \mu' \circ f  =\mu'_{\alpha'} \circ f.
$$
\end{proof}
The theorem (\ref{thmConstrHomAlt}) gives a procedure to construct Hom-alternative algebras
using ordinary alternative algebras and their algebras
endomorphisms.
\begin{remark}\label{HomAltInducedByAlt}
Let  $(V,\mu,\alpha)$ be a Hom-alternative algebra, one may ask whether this
Hom-alternative algebra is   induced by an
ordinary alternative algebra $(V,\widetilde{\mu})$, that is $\alpha$
is an algebra endomorphism with respect to $\widetilde{\mu}$ and
$\mu=\alpha\circ\widetilde{\mu}$. This question was addressed and discussed for
Hom-associative algebras in \cite{FregierGohr2,Gohr}.

First observation, if $\alpha$
is an algebra endomorphism with respect to $\widetilde{\mu}$
then $\alpha$ is also an algebra endomorphism
with respect to $\mu$. Indeed,
$$\mu(\alpha(x),\alpha(y))=\alpha\circ\widetilde{\mu}(\alpha(x),\alpha(y))=
\alpha\circ\alpha\circ\widetilde{\mu}(x,y)=\alpha\circ\mu(x,y).$$

Second observation, if $\alpha$ is bijective then $\alpha^{-1}$ is
also an algebra automorphism. Therefore one may use an untwist
operation on the Hom-alternative algebra in order to recover the
alternative algebra ($\widetilde{\mu}=\alpha^{-1}\circ\mu$).
\end{remark}

\subsection{ Examples of Hom-Alternative algebras}
We construct examples of Hom-alternative using theorem
(\ref{thmConstrHomAlt}).
We use to this end the classification of 4-dimensional alternative algebras
which are not associative (see \cite{EGG}) and the algebra of octonions (see \cite{Baez}).
 For each algebra, algebra endomorphisms are provided. Therefore, Hom-alternative
 algebras are attached according to theorem
(\ref{thmConstrHomAlt}).

\begin{example}
[Hom-alternative algebras of dimension 4]  According to \cite{EGG}, p 144,
there are exactly two alternative but not associative algebras of
dimension 4 over any field.
With respect to a basis $\{e_0, e_1, e_2, e_3\}$, one algebra is given by the
following multiplication (the unspecified products are zeros)
\begin{eqnarray*} && \mu_1(e_0,e_0)=e_0, \; \mu_1 (e_0,e_1)=e_1,
\;\mu_1 (e_2,e_0)=e_2,\\&&
\mu_1 (e_2,e_3)=e_1, \;\mu_1 (e_3,e_0)=e_3, \;\mu_1 (e_3,e_2)=- e_1.
\end{eqnarray*}
The other algebra is given by
\begin{eqnarray*}&& \mu_2 ( e_0,e_0)=e_0, \;\mu_2 ( e_0,e_2)=e_2,
 \;\mu_2 (e_0,e_3)=e_3,
 \\&& \mu_2 (e_1,e_0)=e_1, \;\mu_2 (e_2,e_3)=e_1, \;\mu_2 (e_3,e_2)=- e_1.
 \end{eqnarray*}
These two alternative algebras are anti-isomorphic, that is the first
one is isomorphic to the opposite of the second one. The algebra endomorphisms of $\mu_1$ and  $\mu_2$ are exactly the same.
 We provide two examples of algebra endomorphisms for these algebras.
\begin{enumerate}
 \item The algebra endomorphism  $\alpha_1$ with respect to the same basis is defined by
\begin{eqnarray*}
&& \alpha_1(e_0)= e_0+a_1\ e_1+a_2\ e_2+a_3\ e_3 ,\ \
\alpha_1(e_1)=0,\ \ \\&&
\alpha_1(e_2)=a_4 \ e_2+\frac{a_4 a_3}{a_2} \ e_3,\ \
\alpha_1(e_3)=a_5 \ e_2+\frac{a_5 a_3}{a_2} \ e_3,
\end{eqnarray*}
with $a_1,\cdots,a_5\in \K$ and $a_2\neq 0$.
\item The algebra endomorphism $\alpha_2$ with respect to the same basis is defined by
\begin{eqnarray*}
&& \alpha_2(e_0)= e_0+a_1\ e_1+a_2\ e_2+a_3\ e_3 ,\ \
\alpha_2(e_1)=a_4\ e_1,\ \ \\&&
\alpha_2(e_2)=- \frac{a_4 a_2}{a_5} \ e_2- \frac{a_4 a_3}{a_5} \ e_3,\ \
\alpha_2(e_3)=a_5 \ e_1+a_6 \ e_2+\frac{a_6 a_3-a_5}{a_2} \ e_3,
\end{eqnarray*}
with $a_1,\cdots,a_6\in \K$ and $a_2,a_5\neq 0$.
\end{enumerate}

According to theorem (\ref{thmConstrHomAlt}), the linear map $\alpha_1$ an the
following multiplications 

 \begin{align*} & \bullet  \mu^1_1(e_0,e_0)=e_0+a_1\ e_1+a_2\ e_2+a_3\ e_3, \;
\mu^1_1 (e_0,e_1)=0,
\;\mu^1_1 (e_2,e_0)=a_4 \ e_2+\frac{a_4 a_3}{a_2} \ e_3,\\& \ \ \
\mu^1_1 (e_2,e_3)=0, \;\mu^1_1 (e_3,e_0)=a_5 \ e_2+\frac{a_5 a_3}{a_2} \ e_3, \;\mu_1 (e_3,e_2)=0.
\end{align*}

\begin{align*}& \bullet \mu^1_2 ( e_0,e_0)=e_0+a_1\ e_1+a_2\ e_2+a_3\ e_3, \;
\mu^1_2 ( e_0,e_2)=a_4 \ e_2+\frac{a_4 a_3}{a_2} \ e_3,
 \;\\& \ \ \ \mu^1_2 (e_0,e_3)=a_5 \ e_2+\frac{a_5 a_3}{a_2} \ e_3,
 \; \mu^1_2 (e_1,e_0)=0, \;\mu^1_2 (e_2,e_3)=0, \;\mu^1_2 (e_3,e_2)=0.
 \end{align*}
 determine  4-dimensional Hom-alternative algebras.
 
The linear map $\alpha_2$ leads to the following multiplications

\begin{align*}& \bullet \mu^2_1(e_0,e_0)=e_0+a_1\ e_1+a_2\ e_2+a_3\ e_3, \;
 \mu^2_1 (e_0,e_1)=a_4\ e_1,
\;\mu^2_1 (e_2,e_0)=- \frac{a_4 a_2}{a_5} \ e_2- \frac{a_4 a_3}{a_5} \ e_3,
\\& \ \ \
\mu^2_1 (e_2,e_3)=a_4\ e_1, \;\mu^2_1 (e_3,e_0)=e_3,
\;\mu^2_1 (e_3,e_2)=- a_4\ e_1.
\end{align*}

\begin{align*}& \bullet \mu^2_2 ( e_0,e_0)=e_0+a_1\ e_1+a_2\ e_2+a_3\ e_3,
 \;\mu^2_2 ( e_0,e_2)=- \frac{a_4 a_2}{a_5} \ e_2- \frac{a_4 a_3}{a_5} \ e_3,
 \;\\& \ \ \ \mu^2_2 (e_0,e_3)=a_5 \ e_1+a_6 \ e_2+\frac{a_6 a_3-a_5}{a_2} \ e_3,
 \;\\& \ \ \ \mu^2_2 (e_1,e_0)=a_4\ e_1, \;\mu_2 (e_2,e_3)=a_4\ e_1,
  \;\mu^2_2 (e_3,e_2)=- a_4\ e_1.
 \end{align*}
\end{example}

\begin{example}[Octonions]
{\rm Octonions are typical example of alternative algebra. They were discovered in 1843 by John T. Graves who
called them Octaves and independently by Arthur Cayley in 1845. See \cite{Baez} for
the role of the octonions in algebra, geometry and topology and see
also \cite{Albuquerque} where octonions are viewed as a quasialgebra. The
octonions algebra which is also called Cayley Octaves or Cayley
algebra is an  $8$-dimensional defined with respect to a basis
$\{u,e_1,e_2,e_3,e_4,e_5,e_6,e_7\}$, where $u$ is the identity for the
multiplication, by the following multiplication table.  The table
describes multiplying the $i$th row elements  by the $j$th column
elements.

\[
\begin{array}{|c|c|c|c|c|c|c|c|c|}
  \hline
   \ & u& e_1 & e_2 & e_3 & e_4 & e_5 & e_6 & e_7 \\ \hline
   u& u& e_1 & e_2 & e_3 & e_4 & e_5 & e_6 & e_7 \\ \hline
   e_1 &e_1 & -u & e_4 & e_7 & -e_2 & e_6 & -e_5 &- e_3 \\ \hline
   e_2 &e_2 & -e_4 & -u & e_5 & e_1 & -e_3 & e_7 & -e_6 \\ \hline
   e_3 &e_3 & -e_7 & -e_5 & -u & e_6 & e_2 & -e_4 & e_1 \\ \hline
   e_4 &e_4 & e_2 & -e_1 & -e_6 & -u & e_7 & e_3 & -e_5 \\ \hline
   e_5 &e_5 & -e_6 & e_3 & -e_2 & -e_7 & -u & e_1 & e_4 \\ \hline
   e_6 &e_6 & e_5 & -e_7 & e_4 & -e_3 & -e_1 & -u & e_2 \\ \hline
   e_7 &e_7 & e_3 & e_6 & -e_1 & e_5 & -e_4 & -e_2 & -u \\
    \hline
\end{array}
\]
\

The diagonal algebra endomorphism of octonions are give by maps $\alpha$
defined with respect to the basis $\{u,e_1,e_2,e_3,e_4,e_5,e_6,e_7\}$ by
\begin{eqnarray*}
&& \alpha(u)=u,\ \ \alpha(e_1)=a\ e_1,\ \ \alpha(e_2)=b\ e_2,\ \
\alpha(e_3)=c \ e_3
,\\&&
 \alpha(e_4)=a b\ e_4,\ \ \alpha(e_5)=b c\  e_5
,\ \ \alpha(e_6)=a b c\  e_6,\ \ \alpha(e_7)=a c\ e_7,
\end{eqnarray*}
where $a,b,c$ are any parameter in $\K$.
The associated Hom-alternative algebra to the octonions algebra according to
theorem (\ref{thmConstrHomAlt}) is described by the map $\alpha$ and the 
multiplication   defined by the following table. The table
describes multiplying the $i$th row elements  by the $j$th column
elements.

\[
\begin{array}{|c|c|c|c|c|c|c|c|c|}
  \hline
   \ & u& e_1 & e_2 & e_3 & e_4 & e_5 & e_6 & e_7 \\ \hline
   u& u& a e_1 & b e_2 &c e_3 & ab e_4 & bc e_5 & a bc e_6 & a c e_7 \\ \hline
   e_1 &a e_1 & -u & ab e_4 & a c e_7 & -b e_2 & a bc e_6 & -bc e_5 &- c e_3 \\ \hline
   e_2 &b e_2 & -ab e_4 & -u & bc e_5 & a e_1 & -c e_3 & a c e_7 & -a bc e_6 \\ \hline
   e_3 &c e_3 & -a c e_7 & -bc e_5 & -u & a bc e_6 & b e_2 & -ab e_4 & a e_1 \\ \hline
   e_4 &ab e_4 & b e_2 & -a e_1 & -a bc e_6 & -u & a c e_7 & c e_3 & -bc e_5 \\ \hline
   e_5 &bc e_5 & -a bc e_6 & c e_3 & -b e_2 & -a c e_7 & -u & a e_1 & ab e_4 \\ \hline
   e_6 &a bc e_6 & bc e_5 & -a c e_7 & ab e_4 & -c e_3 & -a e_1 & -u & b e_2 \\ \hline
   e_7 &a c e_7 & c e_3 & a bc e_6 & -a e_1 & bc e_5 & -ab e_4 & -b e_2 & -u \\
    \hline
\end{array}
\]
Notice that the new algebra is no longer unital, neither  an
alternative algebra since
$$\mu(u,\mu(u,e_1))-\mu(\mu(u,u),e_1)=(a^2-a)e_1,
$$
which is different from $0$ when $a\neq 0,1$.
}
\end{example}


\section{Hom-Jordan algebras}\label{SectionHomJordan}
In this section, we introduce a generalization of Jordan algebra by twisting
the usual Jordan identity
\begin{equation}\label{JordanIdentity}
(x\cdot y)\cdot x^2=x\cdot( y\cdot x^2).
\end{equation}
We show that this
generalization fits with Hom-associative algebras. Also, we provide a procedure to
construct examples starting from an ordinary Jordan algebras.
\begin{definition}
  A \textbf{Hom-Jordan algebra} is a triple $(V, \mu, \alpha)$
    consisting of a linear space $V$, a bilinear map $\mu: V\times V
    \rightarrow V$ which is commutative and a  homomorphism $\alpha: V \rightarrow V$
satisfying
\begin{equation}\label{HomJordanIdentity}
\mu(\alpha^2(x),\mu (y,\mu(x,x)))=\mu (\mu (\alpha (x),y),\alpha
(\mu(x,x)))
\end{equation}
where $\alpha^2=\alpha \circ \alpha$.
\end{definition}

Let $(V, \mu, \alpha)$ and $(V', \mu', \alpha')$ be two Hom-Jordan algebras. A linear map $\phi\ :V\rightarrow V^{\prime }$ is a morphism
of Hom-Jordan algebras if%
$$
\mu ^{\prime }\circ (\phi\otimes \phi)=\phi\circ \mu \quad \text{
and } \qquad \phi\circ \alpha=\alpha^{\prime }\circ \phi.
$$
\begin{remark}
Since the multiplication is commutative, one may write the identity (\ref{HomJordanIdentity})
as
\begin{equation}\label{HomJordanIdentity2}
\mu(\mu (y,\mu(x,x)),\alpha^2(x))=\mu (\mu (y,\alpha (x)),\alpha
(\mu(x,x))).
\end{equation}
\end{remark}
When the twisting map $\alpha$ is the identity map, we recover the classical
notion of Jordan algebra. 

The identity (\ref{HomJordanIdentity}) is motivated by the following
functor which associates to a Hom-associative algebra a Hom-Jordan
algebra by polarization.

\begin{theorem}
To any Hom-associative algebra $(V,m,\alpha)$ defined by the
multiplication $m$ and a homomorphism $\alpha$ over a $\K$-linear
space $V$, one may associate a  Hom-Jordan algebra
$(V,\mu,\alpha)$, where the multiplication
$\mu$ is  defined for all $x,y \in V$ by
$$
\mu( x,y )=\frac{1}{2}(m (x,y)+m (y,x )).
$$
\end{theorem}
\begin{proof}
The commutativity of $\mu$ is obvious. We   compute the difference
$$D=\mu(\alpha^2(x),\mu
(y,\mu(x,x)))-\mu (\mu
(\alpha(x),y),\alpha (\mu(x,x)))
$$
A straightforward computation gives
\begin{eqnarray*}
D&=&
m(\alpha^2(x),m(y,m (x,x)))
+m(m(y,m(x,x)),\alpha^2(x))+
m(\alpha^2(x),m(m(x,x),y))
\\ \ &&
+m(m(m(x,x),y),\alpha^2(x))
-m(m(\alpha(x),y),\alpha(m(x,x)))
-m(\alpha(m(x,x)),m(\alpha(x),y))
\\ \ &&
-m(m(y,\alpha(x)),\alpha(m(x,x)))
-m(\alpha(m(x,x)),m(y,\alpha(x))).
\end{eqnarray*}
We have by Hom-associativity
\begin{eqnarray*}
m(\alpha^2(x),m(y,m (x,x)))-m(m(\alpha(x),y),\alpha(m(x,x)))=0\\
m(m(m(x,x),y),\alpha^2(x))-m(\alpha(m(x,x)),m(y,\alpha(x)))=0.
\end{eqnarray*}
Therefore
\begin{eqnarray*}
D&=&
m(m(y,m(x,x)),\alpha^2(x))+
m(\alpha^2(x),m(m(x,x),y))
\\ \ &&
-m(\alpha(m(x,x)),m(\alpha(x),y))
-m(m(y,\alpha(x)),\alpha(m(x,x))).
\end{eqnarray*}
One may show that for any Hom-associative algebra we have
\begin{eqnarray*}
m(\alpha(m(x,x)),m(\alpha(x),y))
&= m(m(m(x,x),\alpha(x)),\alpha(y))
\\&=m(m(\alpha(x),m(x,x)),\alpha(y))
\\&=m(\alpha^2(x),m(m(x,x),y)),
\end{eqnarray*}
and similarly
\begin{equation*}
m(m(y,\alpha(x)),\alpha(m(x,x)))=m(m(y,m(x,x)),\alpha^2(x)).
\end{equation*}
Thus
\begin{eqnarray*}
D&=&
m(m(y,m(x,x)),\alpha^2(x))+
m(\alpha^2(x),m(m(x,x),y))
\\ \ &&
-m(\alpha^2(x),m(m(x,x),y))
-m(m(y,m(x,x)),\alpha^2(x))
\\ \ &=&0.
\end{eqnarray*}
\end{proof}
\begin{remark}
The definition of Hom-Jordan algebra seems to be non natural one
expects that the identity should be of the form
\begin{equation}\label{HomJordanIdentity2}
\mu(\alpha(x),\mu (y,\mu(x,x))))=\mu (\mu (x,y),\alpha (\mu(x,x)))
\end{equation}
or
\begin{equation}\label{HomJordanIdentity3}
\mu(\alpha(x),\mu (y,\mu(x,x))))=\mu (\mu (x,y),\mu(x,\alpha (x))).
\end{equation}
It turns out that these identities do not fit with the previous
proposition.

 Notice also that in general a Hom-alternative algebra doesn't lead to a
 Hom-Jordan algebra.
\end{remark}

The following theorem gives a procedure to construct Hom-Jordan algebras
using ordinary Jordan algebras and their algebra
endomorphisms.
\begin{theorem}\label{thmConstrHomJordan}
Let $(V,\mu)$ be a Jordan algebra and $\alpha : V\rightarrow V$ be an
 algebra endomorphism. Then $(V,\mu_\alpha,\alpha)$,
where $\mu_\alpha=\alpha\circ\mu$,  is a  Hom-Jordan
algebra.

Moreover, suppose that  $(V',\mu')$ is another
Jordan algebra  and $\alpha ' : V'\rightarrow V'$ is an algebra
endomorphism. If $f:V\rightarrow V'$ is an algebras morphism that
satisfies $f\circ\alpha=\alpha'\circ f$ then
$$f:(V,\mu_\alpha,\alpha)\longrightarrow (V',\mu'_{\alpha '},\alpha ')
$$
is a morphism of  Hom-Jordan algebras.
\end{theorem}
\begin{proof}
We show that $(V,\mu_\alpha,\alpha)$ satisfies the
Hom-Jordan identity (\ref{HomJordanIdentity}) while $(V,\mu)$ satisfies
the Jordan identity (\ref{JordanIdentity}). Indeed
\begin{align*}
& \mu_\alpha(\alpha^2(x),\mu_\alpha
(y,\mu_\alpha(x,x)))-\mu_\alpha (\mu_\alpha
(\alpha(x),y),\alpha (\mu_\alpha(x,x)))
\\ &  = \alpha\circ \mu(\alpha^2(x),\alpha\circ \mu
(y,\alpha\circ \mu(x,x)))-\alpha\circ \mu (\alpha\circ \mu
(\alpha(x),y),\alpha^2\circ \mu(x,x))
\\ &  = \alpha^2( \mu(\alpha(x), \mu
(y,\alpha\circ \mu(x,x)))- \mu ( \mu
(\alpha(x),y),\alpha\circ \mu(x,x)))
\\ &  = \alpha^2( \mu(\alpha(x), \mu
(y,\mu(\alpha(x),\alpha(x))))- \mu ( \mu
(\alpha(x),y), \mu(\alpha(x),\alpha(x))))
\\ &  =0.
\end{align*}

the second assertion follows from
$$f\circ \mu_\alpha=f\circ \alpha \circ \mu=\alpha'\circ
f \circ \mu =\alpha'\circ \mu' \circ f  =\mu'_{\alpha'} \circ f.
$$
\end{proof}
\begin{remark}
We may give here similar observations as in the remark
(\ref{HomAltInducedByAlt})
concerning Hom-Jordan algebra  induced by an
ordinary Jordan algebra.
\end{remark}
We provide in the sequel example of Hom-Jordan algebras.

\begin{example}
We consider Hom-Jordan algebras associated to  Hom-associative
 algebras described in example (\ref{example1ass}).
Let $\{e_1,e_2,e_3\}$  be a basis of a $3$-dimensional linear space
$V$ over $\K$. The following multiplication $\mu$ and linear map
$\alpha$ on $V$ define  Hom-Jordan algebras over $\K^3${\rm :}
$$
\begin{array}{ll}
\begin{array}{lll}
 \widetilde{\mu} ( e_1,e_1)&=& a\ e_1, \ \\
\widetilde{\mu} ( e_1,e_2)&=&\widetilde{\mu} ( e_2,e_1)=a\ e_2,\\
\widetilde{\mu} ( e_1,e_3)&=&\widetilde{\mu} ( e_3,x_1)=b\ e_3,\\
 \end{array}
 & \quad
 \begin{array}{lll}
\widetilde{\mu} ( e_2,e_2)&=& a\ e_2, \ \\
\widetilde{\mu} ( e_2, e_3)&=& \frac{1}{2}b\ e_3, \ \\
\widetilde{\mu} ( e_3,e_2)&=& \widetilde{\mu} ( e_3,e_3)=0,
  \end{array}
\end{array}
$$

$$  \alpha (e_1)= a\ e_1, \quad
 \alpha (e_2) =a\ e_2 , \quad
   \alpha (e_3)=b\ e_3
$$
where $a,b$ are parameters in $\K$.

It turns out that the multiplication of this Hom-Jordan algebra defines
 a Jordan algebra.

\end{example}

\begin{remark}
We may define the noncommutative Jordan algebras as triples $(V,\mu,\alpha)$ satisfying the identity (\ref{HomJordanIdentity}) and the flexibility condition, which is a generalization of the commutativity. Eventually, we may consider the Hom-flexibilty defined by the identity
$\mu(\alpha(x), \mu (y, x))= \mu (\mu (x, y), \alpha (x)).$
 
\end{remark}

\end{document}